\documentclass[a4paper,11pt]{article}

\usepackage{amsmath}
\usepackage{amssymb}
\usepackage{amsthm}
\usepackage{graphicx}
\usepackage{amscd}
\usepackage{epic, eepic}
\usepackage{url}
\usepackage{color}
\usepackage[utf8]{inputenc} 
\usepackage{comment}
\usepackage{enumerate}   
\usepackage{graphicx}
\usepackage{epstopdf}
\usepackage{enumitem}
\usepackage{tikz}
\usepackage[top=24mm, bottom=25mm, left=23mm, right=23mm]{geometry}
\usepackage[bookmarks=false,hidelinks]{hyperref}
\usepackage[none]{hyphenat}

\newtheorem{theorem}{\bf Theorem}[section]
\newtheorem{lemma}[theorem]{\bf Lemma}
\newtheorem{corollary}[theorem]{\bf Corollary}
\newtheorem{proposition}[theorem]{\bf Proposition}
\newtheorem{conjecture}[theorem]{\bf Conjecture}
\newtheorem{problem}[theorem]{\bf Problem}

\theoremstyle{definition}
\newtheorem*{claim}{Claim}
\newtheorem*{remark}{\bf Remark}
\newtheorem{definition}[theorem]{\bf Definition}

\def\eps{\varepsilon}

\def\cT{\mathcal{T}}

\title{\vspace{-0.9cm} Small subgraphs with large average degree}

\author{
	Oliver Janzer\thanks{Department of Mathematics, ETH, Z\"urich, Switzerland. Research supported by an ETH Z\"urich Postdoctoral Fellowship 20-1 FEL-35. Email: \textbf{oliver.janzer@math.ethz.ch}.}
	\and
	Benny Sudakov\thanks{Department of Mathematics, ETH, Z\"urich, Switzerland. Research supported in part by SNSF grant 200021\_196965. Email: \textbf{\{benjamin.sudakov, istvan.tomon\}@math.ethz.ch}.}
	\and
	Istv\'{a}n Tomon\footnotemark[2]
}

\date{}

\begin{document}
	
	\maketitle
	\begin{abstract}
		In this paper we study the fundamental problem of finding small dense subgraphs in a given graph. For a real number $s>2$, we prove that every graph on $n$ vertices with average degree at least $d$ contains a subgraph of average degree at least $s$ on at most $nd^{-\frac{s}{s-2}}(\log d)^{O_s(1)}$ vertices. This is optimal up to the polylogarithmic factor, and resolves a conjecture of Feige and Wagner. In addition, we show that every graph with $n$ vertices and average degree at least $n^{1-\frac{2}{s}+\varepsilon}$ contains a subgraph of average degree at least $s$ on $O_{\varepsilon,s}(1)$ vertices, which is also optimal up to the constant hidden in the $O(.)$ notation, and resolves a conjecture of Verstra\"ete.
		
	\end{abstract}
	
	\section{Introduction}
	Given a graph $G$ and a parameter $k$, the \emph{densest $k$-subgraph problem} asks to find a $k$-vertex subgraph of $G$ of maximum average degree. This is one of the central problems in theoretical computer science. It is NP-hard, and has no polynomial-time approximation scheme (PTAS) under certain complexity theoretic assumptions \cite{F02,K04}. On the other hand, the currently best known approximation algorithm achieves an $O(n^{1/4})$-approximation \cite{BCCFV}. 
	
	In addition to the algorithmic perspective, another natural direction for the above problem is to understand the maximum density of the small subgraphs of a given graph which one can theoretically guarantee. The precise problem under consideration, proposed by Feige and Wagner \cite{FW16}, is that given a   positive integer $n$ and real numbers $d,s$ satisfying $d\geq s\geq 2$, what is the minimum of $t=t(n,d,s)$ such that every graph $G$ on $n$ vertices with average degree at least $d$ contains a subgraph of average degree at least $s$ on at most $t$ vertices. This question also falls squarely within the context of the so called local-global principle, that states that one can obtain global understanding of a structure from having a good understanding of its local properties, or vice versa. This phenomenon has been ubiquitous in many areas of mathematics and beyond, see e.g. \cite{local-global-2,BS,gromov,local-global-1}.
	
	The question of Feige and Wagner, in the case $s=2$, is equivalent to the famous \emph{girth problem}, that asks for the length of the shortest cycle in a graph on $n$ vertices with average degree $d$. This problem is extensively studied, and using our notation, it is well known that $t(n,d,2)=\Theta(\log_{d-1}n)$ (see e.g. \cite{B04}, page 104, Theorems 1.1 and 1.2). However, it is a major open problem to determine the leading coefficient.	Much less is known if $s>2$. A simple probabilistic argument gives the following result.
	
	\begin{proposition} \label{prop:lower bound feige}
		For every $s>2$, there is a positive $c_s$ such that for all $s\leq d\leq n-1$, we have $t(n,d,s)\geq c_s nd^{-\frac{s}{s-2}}$. In other words, for every $s\leq d\leq n-1$, there is an $n$-vertex graph $G$ with average degree at least $d$ in which every subgraph with average degree at least $s$ has at least $c_snd^{-\frac{s}{s-2}}$ vertices.
	\end{proposition}

	Feige and Wagner \cite{FW16} proposed the conjecture that this lower bound on $t(n,d,s)$ is optimal, up to polylogarithmic factors in $n$. In the special case  $s\approx 4$, they also proved certain results in the support of it. First, they showed that if $\varepsilon>0$ and $s=4-\varepsilon$, then $t(n,d,s)=O_{\varepsilon}(nd^{-2})$. Second, they proved that $t(n,d,4)=O_{\varepsilon}(nd^{-1.8+\varepsilon})$ for every $\varepsilon>0$. Also, Alon and Hod (personal communication) proved the aforementioned conjecture for certain special values of $s$ and a limited range of $d$.  Here, we completely settle the conjecture of Feige and Wagner with the following theorem (which is even stronger as the error term is logarithmic in $d$ instead of $n$).
	
	\begin{theorem} \label{thm:polylog_average_deg}
		For every $s>2$, there is a constant $C=C(s)$ such that the following holds for every sufficiently large $d$. Let $G$ be an $n$-vertex graph with average degree at least $d$, where $d\leq n^{\frac{s-2}{s}}$. Then there is a non-empty set $R\subset V(G)$ of size at most $nd^{-\frac{s}{s-2}}(\log d)^C$ such that $G[R]$ has average degree at least $s$.
	\end{theorem}
	
	\noindent
	While our proof of this result is non-algorithmic, it gives the best theoretical lower bound on the average degree one is guaranteed to find. It would be interesting to decide whether there is a polynomial time algorithm that finds a subgraph that achieves the bound provided by the above theorem.
	
	Theorem \ref{thm:polylog_average_deg} cannot be used to find a constant sized (independent of $n$) subgraph with large average degree. In this case, one cannot expect a similar answer as before, as the random deletion method shows the following.
	
	\begin{proposition}\label{prop:lower_bound2}
		For every $s>2$ and positive integer $t$, there exists  $\varepsilon=\varepsilon(s,t)>0$ such that the following holds for every sufficiently large $n$. There exists a graph $G$ on $n$ vertices with average degree at least $n^{1-\frac{2}{s}+\varepsilon}$ such that every subgraph of $G$ on at most $t$ vertices has average degree less than $s$.
	\end{proposition}
	
	A similar argument shows that in case $d=\Omega\big(n^{\frac{s-2}{s}}\big)$, the logarithmic error term is indeed needed in Theorem \ref{thm:polylog_average_deg}. Motivated by applications from \cite{NV08} to parity check matrices, Verstr\"aete (see \cite{JN17}) conjectured that this lower bound presented in Proposition \ref{prop:lower_bound2} is optimal in a certain sense. More precisely, he proposed the conjecture that for every $s>2$ and $\varepsilon>0$ there exists some $t=t(s,\varepsilon)$ such that every graph on $n$ vertices with average degree at least $n^{1-\frac{2}{s}+\varepsilon}$ must contain a subgraph on at most $t$ vertices with average degree at least $s$. In the special case $s$ is an integer, this was proved by Jiang and Newman \cite{JN17}. Janzer \cite{J20} strengthened this result, obtaining under the same hypothesis an $s$-regular subgraph. More precisely he proved that if $G$ is a graph on $n$ vertices with at least $n^{2-\frac{1}{r}+\frac{1}{k+r-1}+\varepsilon}$ edges for sufficiently large $n$, then $G$ contains an $r$-blowup of the cycle $C_{2k}$ (note that the $r$-blowup of the cycle $C_{2k}$ is $s=2r$-regular and by taking $k$ large one can make the term $\frac{1}{k+r-1}$ arbitrarily small). In our next theorem, we prove the conjecture of Verstra\"ete for all real values of $s>2$.
	
	\begin{theorem} \label{thm:bounded_size_average_deg}
		For every $s>2$ and $\eps>0$, there is a positive integer $t$ such that the following holds for all sufficiently large $n$. Let $G$ be an $n$-vertex graph of average degree $d\geq n^{1-\frac{2}{s}+\eps}$. Then there is a non-empty set $R\subset V(G)$ of size at most $t$ such that $G[R]$ has average degree at least~$s$.
	\end{theorem}
	
	Our results are closely related to the problem of Erd\H{o}s, Faudree, Rousseau and Schelp on
	finding small subgraphs of large minimum degree. In \cite{EFRS90}, they determined the minimal number of edges in a graph on $n$ vertices which guarantees a \emph{proper} subgraph (i.e., with $u<n$ vertices) of minimum degree at least $s$ (see \cite{S19} for additional details and recent developments). Erd\H os, Faudree, Rousseau and Schelp \cite{EFRS91} further asked the following general question. Given positive integers $n$ and $s$, and a positive real number $d$ satisfying $d\geq s\geq 2$, what is the minimum of $u=u(n,d,s)$ such that every graph $G$ on $n$ vertices with average degree at least $d$ contains a subgraph of minimum degree at least $s$ on at most $u$ vertices? It is reasonable to suspect that $u(n,d,s)\approx t(n,d,s)$, that is, that Theorems~\ref{thm:polylog_average_deg} and~\ref{thm:bounded_size_average_deg} hold with the average degree of $G[R]$ replaced with its minimum degree. In case $s$ is even, the minimum degree version of Theorem \ref{thm:bounded_size_average_deg} does hold by the aforementioned results of \cite{J20}  and \cite{JN17}, the first of which even guarantees a regular subgraph. Moreover, the case $s=3$ follows from a recent result of Janzer \cite{J21}, which refutes a conjecture of Erd\H{o}s and Simonovits \cite{E81} (and again provides a regular subgraph). However, the cases when $s$ is odd and greater than 3 remain open. On the other hand, the minimum degree variant of Theorem \ref{thm:polylog_average_deg} is completely open for every $s\geq 3$, and our methods do not seem to be adaptable for this problem. At least, by noting that every graph of average degree at least $2s$ contains a subgraph of minimum degree at least $s$ (see Lemma~\ref{lemma:mindeg}), we get the following immediate corollary of Theorem \ref{thm:polylog_average_deg}.
	
	\begin{corollary}
		For every integer $s\geq  2$, there is a constant $C=C(s)$ such that the following holds for every sufficiently large $d$. Let $G$ be an $n$-vertex graph with average degree at least $d$, where $d\leq n^{\frac{s-1}{s}}$. Then there is a non-empty set $R\subset V(G)$ of size at most $nd^{-\frac{s}{s-1}}(\log d)^C$ such that $G[R]$ has minimum degree at least $s$.
	\end{corollary}
	
	\section{Small subgraphs of large average degree}
	
	In this section, we prove Theorems \ref{thm:polylog_average_deg} and \ref{thm:bounded_size_average_deg}. Both proofs follow the same argument, however, with a different range of parameters. Let us give a brief outline of this argument.
	
	The key idea is that for every rational number $\rho>1$ we construct a tree $T$, which we refer to as a \emph{balanced tree}, with the following property. Let $q$ be the number of leaves of $T$. If $H$ is a graph that is the union of copies of $T$ having the same set of leaves, then the average degree of $H$ is at least $2\rho(1-\frac{q}{|V(H)|})$. 
	
	Now suppose we are given a graph $G$ with $n$ vertices and average degree at least $d$,  which does not contain a subgraph of order at most $t\approx nd^{-\frac{s}{s-2}}$ and average degree at least $s$. Take $\rho$ such that $2\rho$ is slightly larger than $s$, let $T$ be the balanced tree with respect to $\rho$, and let $q$ be the number of leaves of $T$. By counting the number of subgraphs of $G$ isomorphic to $T$ and using the pigeonhole principle, we find a large collection $\mathcal{T}$ of copies of $T$ in $G$, all having the same set of leaves. Let $H$ be their union. We show that some subgraph of $H$ will contradict our assumption on $G$, i.e. it has order at most $t$ and average degree at least $s$. In order to show this, we consider two cases depending on  the number of vertices in $H$. If $|V(H)|>t$, we take some sub-collection $\mathcal{T}'$ of $\mathcal{T}$ such that $H'$, which denotes the union of the copies of $T$ in $\mathcal{T}'$,  has order roughly $t$. Our choice of parameters will guarantee that $2\rho(1-\frac{q}{|V(H')|})\geq s$, so $H'$ suffices by the above mentioned property of $T$. Otherwise, we argue that unless $H$ has average degree at least $s$, it cannot contain the described number of copies of $T$, and we are done again.
	
	\subsection{Preliminaries}
	
	In this section, we prove the lower bounds (Propositions \ref{prop:lower bound feige} and \ref{prop:lower_bound2}) and collect some basic results. First, let us start with the following simple consequence of the multiplicative Chernoff bound.
	
	\begin{lemma}\label{lemma:chernoff}
		Let $X$ be the sum of independent indicator (i.e. 0-1 valued) random variables. Then $\mathbb{P}(X\leq \frac{\mathbb{E}(X)}{2})<e^{-\frac{\mathbb{E}(X)}{8}}$.
	\end{lemma}
	
	Next, we present the promised probabilistic lower bound arguments. 
	
	\begin{proof}[Proof of Proposition \ref{prop:lower bound feige}]
		Let $c_s$ be sufficiently small. If $d\geq \frac{n-1}{2}$, then $c_snd^{-\frac{s}{s-2}}<1$, so we can take $G$ to be any $n$-vertex graph with average degree at least $d$.
		
		Else, let $G$ be a random graph on $n$ vertices in which each edge is chosen with probability $p=\frac{2d}{n-1}$, independently of all other edges. Then $|E(G)|$ is the sum of independent indicator random variables and has mean $nd$, so by Lemma \ref{lemma:chernoff}, the probability that $G$ has fewer than $nd/2$ edges (i.e., that $G$ has average degree less than $d$) is at most $\exp(-\frac{nd}{8})\leq e^{-\frac{1}{8}}\leq \frac{99}{100}$.
		
		Let $R$ be a subset of $V(G)$ of size $r\leq c_s nd^{-\frac{s}{s-2}}$. By the union bound, the probability that $G[R]$ has average degree at least $s$ (i.e., that $G[R]$ has at least $\frac{rs}{2}$ edges) is at most $\binom{\binom{r}{2}}{\lceil rs/2\rceil}p^{\lceil \frac{rs}{2} \rceil}\leq \left(\frac{e\binom{r}{2}}{\lceil rs/2\rceil} p \right)^{\lceil \frac{rs}{2}\rceil}\leq (erp)^{\frac{rs}{2}}$. Hence, by the union bound, the probability that $G$ has a subgraph of average degree at least $s$ on at most $c_s n d^{-\frac{s}{s-2}}$ vertices is at most
		\begin{equation}
			\sum_{r=1}^{\lfloor c_s n d^{-\frac{s}{s-2}}\rfloor} \binom{n}{r}(erp)^{\frac{rs}{2}}\leq \sum_{r=1}^{\lfloor c_s n d^{-\frac{s}{s-2}}\rfloor} \left(\frac{en}{r}\cdot(erp)^{\frac{s}{2}}\right)^r= \sum_{r=1}^{\lfloor c_s n d^{-\frac{s}{s-2}}\rfloor} (e^{\frac{s}{2}+1}r^{\frac{s-2}{2}}np^{\frac{s}{2}})^r. \label{eqn:union bound}
		\end{equation}
		Moreover, for $r\leq \lfloor c_s n d^{-\frac{s}{s-2}}\rfloor$, we have
		$$r^{\frac{s-2}{2}}np^{\frac{s}{2}}\leq c_s^{\frac{s-2}{2}} n^{\frac{s-2}{2}} d^{-\frac{s}{2}} n p^{\frac{s}{2}}\leq c_s^{\frac{s-2}{2}} n^{\frac{s-2}{2}} d^{-\frac{s}{2}} n \left(\frac{4d}{n}\right)^{\frac{s}{2}}=c_s^{\frac{s-2}{2}}4^{\frac{s}{2}}.$$
		Hence, if $c_s$ is sufficiently small, then by (\ref{eqn:union bound}), the probability that $G$ has a subgraph of average degree at least $s$ on at most $c_snd^{-\frac{s}{s-2}}$ vertices is at most $\frac{1}{1000}$. It follows that with positive probability $G$ has no such subgraph but has average degree at least $d$, completing the proof.
	\end{proof}
	
	\begin{proof}[Proof of Proposition \ref{prop:lower_bound2}]
		We show that $\varepsilon= \frac{1}{ts}$ suffices. Assume that $n$ is sufficiently large with respect to $s$ and $t$, and let $G'$ be the graph on $n$ vertices in which each edge is present independently with probability $p=n^{-\frac{2}{s}+\frac{2}{ts}}$. Letting $X$ be the number of edges of $G'$, we have $\mathbb{E}(X)=\binom{n}{2}p>\frac{1}{4}n^{2-\frac{2}{s}+\frac{2}{ts}}$.
		
		Let $\mathcal{R}$ be the family of graphs $R\subset V(G')^{(2)}$ with $r\leq t$ vertices and exactly $\lceil \frac{rs}{2}\rceil$ edges. Say that such an $R$ is \emph{bad} if $R$ is a subgraph of $G'$. Clearly, $\mathbb{P}(R\mbox{ is bad})=p^{\lceil \frac{rs}{2}\rceil}$. 
		Let $Y$ be the number of bad elements of $\mathcal{R}$, then 
		\begin{align*}
			\mathbb{E}(Y)=&\sum_{R\in\mathcal{R}}\mathbb{P}(R\mbox{ is bad})\leq  \sum_{r=1}^{t}\binom{n}{r}\binom{\binom{r}{2}}{\lceil rs/2\rceil}p^{\lceil \frac{rs}{2}\rceil}\\
			<&\sum_{r=1}^{t} n^{r}(erp)^{\frac{rs}{2}}<t\big(e^{\frac{s}{2}} t^{\frac{s}{2}} n p^{\frac{s}{2}}\big)^{t}=t^{\frac{ts}{2}+1}e^{\frac{ts}{2}}n.
		\end{align*}
		Hence, we have $\mathbb{E}(X-Y)>\frac{1}{8}n^{2-\frac{2}{s}+\frac{2}{ts}}>n^{2-\frac{2}{s}+\varepsilon}$. But then there exists a choice for $G'$ such that $X-Y>n^{2-\frac{2}{s}+\varepsilon}$. For each bad $R\in\mathcal{R}$, remove an edge of $G'$ contained in $R$, and let  the resulting graph be $G$. Then $G$ contains no element of $\mathcal{R}$ as a subgraph, so every subgraph of $G$ on at most $t$ vertices has average degree less than $s$.  Furthermore, $G$ has average degree at least $\frac{2(X-Y)}{n}>n^{1-\frac{2}{s}+\varepsilon}$, finishing the proof.
	\end{proof}
	
	Finally, before we embark on the proofs of our main theorems, let us state a useful lemma about subgraphs of large minimum degree.
	
	\begin{lemma}\label{lemma:mindeg}
		Every graph $G$ of average degree $d$ contains a nonempty subgraph of minimum degree at least $\frac{d}{2}$.
	\end{lemma}
	
	\begin{proof}
		Keep removing vertices of degree less than $\frac{d}{2}$ as long as there is such a vertex. In total, we removed less than $\frac{|V(G)|d}{2}=|E(G)|$ edges, so the resulting graph is nonempty and has minimum degree at least $\frac{d}{2}$.
	\end{proof}
	
	\subsection{Balanced trees}
	In this section, we define \emph{balanced trees}, which one can view as the building blocks of our small subgraph of large average degree. Interestingly, but perhaps not unexpectedly, these trees coincide with the balanced trees constructed by Bukh and Conlon \cite{BC18} in their celebrated paper on the Rational Exponents conjecture.
	
	\begin{definition}
		Let $T$ be a tree with leaf set $L$. For any non-empty $S\subset V(T)\setminus L$, let $$\rho_T(S):=\frac{e_S}{|S|},$$ where $e_S$ is the number of edges in $T$ incident to at least one vertex from $S$. Also, set $\rho_T=\rho_T(V(T)\setminus L)$. We say that $T$ is \emph{balanced} if $\rho_T\leq \rho_T(S)$ holds for every non-empty $S\subset V(T)\setminus L$.
	\end{definition}
	
	The main reason balanced trees are useful for us is given by the following simple lemma which one can prove by induction.
	
	\begin{lemma}[Bukh--Conlon \cite{BC18}] \label{lem:union of balanced trees}
		Let $T$ be a balanced tree with $q$ leaves and let $H$ be a graph which is the union of copies of $T$ with the same set of leaves. Then $e(H)\geq (|V(H)|-q)\rho_T$.
	\end{lemma}
	
	\begin{proof}
		Let $H$ be the union of $k$ copies of $T$ with the same set $L$ of leaves. We prove the inequality $e(H)\geq (|V(H)|-q)\rho_T$ by induction on $k$. If $k=1$, then $e(H)=e(T)\geq e_{V(T)\setminus L}=|V(T)\setminus L|\rho_T=(|V(H)|-q)\rho_T$, as desired. Assume now that $k\geq 2$. Let $T_0$ be one of the $k$ copies of $T$ constituting $H$ and let $H'$ be the union of the remaining $k-1$ copies of $T$. By the induction hypothesis, we have $e(H')\geq (|V(H')|-q)\rho_T$. Let $S=V(T_0)\setminus V(H')$. Note that since all copies of $T$ in $H$ have the same set $L$ of leaves, we have $S\subset V(T_0)\setminus L$. Now observe that $e_S\geq |S|\rho_T$ (where $T_0$ is identified with $T$ and $S$ is viewed as a subset of $V(T)$). Indeed, the inequality is trivial if $S=\emptyset$ and else $e_S=|S|\rho_T(S)\geq |S|\rho_T$ since $T$ is balanced. But then $$e(H)\geq e(H')+e_S\geq (|V(H')|-q)\rho_T+|S|\rho_T=(|V(H)|-q)\rho_T,$$ completing the induction step.
	\end{proof}
	
	Next we describe a construction of balanced trees which are caterpillars. A \emph{caterpillar} is a tree in which the non-leaf vertices form a path.  
	
	\begin{definition}[Bukh--Conlon \cite{BC18}]
		Suppose that $a$ and $b$ are positive integers satisfying $a+1 \leq b < 2a+1$ and put $i = b-a$. We define a tree $T_{a,b}$ by taking a path with $a$ vertices, which are labelled in order as $1,2,\dots,a$, and then adding a leaf to each of the $i+1$ vertices
		$$1, \left\lfloor 1+\frac{a}{i}\right\rfloor,\left\lfloor 1+2\cdot \frac{a}{i}\right\rfloor,\dots,\left\lfloor 1+(i-1)\cdot \frac{a}{i}\right\rfloor,a.$$
		(Note that if $b=2a$, then $\lfloor 1+(i-1)\cdot \frac{a}{i}\rfloor=a$. In this case we attach two leaves in total to vertex $a$.)
		For $b \geq 2a+1$, we define $T_{a,b}$ recursively to be the tree obtained by attaching a leaf to each non-leaf of $T_{a,b-a}$.
	\end{definition}
	
	\begin{remark}
		In \cite{BC18}, trees $T_{a,b}$ for $b\in \{a-1,a\}$ are introduced as well and they are used to define $T_{a,b}$ for $b\in \{2a-1,2a\}$, but one can easily see that our definition gives the same graphs.
	\end{remark}
	
	Bukh and Conlon showed that, indeed, $T_{a,b}$ is balanced for every $a<b$. Combined with the simple observation that $T_{a,b}$ has maximum degree at most $\lceil b/a\rceil+1$, we obtain the following result.
	
	\begin{lemma}[Bukh--Conlon \cite{BC18}] \label{lem:exists balanced tree}
		For any positive integers $a<b$, $T_{a,b}$ is a balanced caterpillar with $a$ non-leaf vertices, $b$ edges and maximum degree at most $\lceil b/a\rceil+1$.
	\end{lemma}
	
	\subsection{Counting trees}
	
	In this section, we provide lower and upper bounds on the number of copies of a fixed tree in graphs with some prescribed properties. Let us start with the lower bound.

	For a graph $G$ and a set $S$ of vertices in $G$, we write $\Gamma_G(S)$ for the set of vertices in $G$ which have a neighbour in $S$. We make use of the following celebrated theorem of Friedman and Pippenger \cite{FP87} about large bounded degree trees in expanding graphs.
	
	\begin{theorem}[Friedman--Pippenger \cite{FP87}] \label{thm:FP}
		If $G$ is a non-empty graph such that for every $S\subset V(G)$ with $|S|\leq 2m-2$, we have $|\Gamma_G(S)|\geq (k+1)|S|$, then $G$ contains every tree with at most $m$ vertices and maximum degree at most $k$.
	\end{theorem}
	
	Say that a graph $G$ is \emph{$(\rho,r)$-sparse} if for every $R\subset V(G)$ of size at most $r$, the number of edges in $G[R]$ is at most $\rho|R|$. Next, we show that $(\rho,r)$-sparse graphs of large minimum degree have good expansion properties.
	
	\begin{lemma} \label{lem:find one tree}
		Let $G$ be a $(\rho,r)$-sparse graph with average degree at least $4\rho (k+2)$. Then $G$ contains every tree with at most $\frac{r}{2(k+2)}$ vertices and maximum degree at most $k$.
	\end{lemma}
	
	\begin{proof}
		By Lemma \ref{lemma:mindeg}, $G$ contains a subgraph $G'$ of minimum degree at least $2\rho (k+2)$. Note that $G'$ is also $(\rho,r)$-sparse. We show that $G'$ already contains every tree with at most $m=\frac{r}{2(k+2)}$ vertices and maximum degree at most $k$. Otherwise, by Theorem \ref{thm:FP}, there is a set $S\subset V(G')$ of size at most $2m-2\leq \frac{r}{k+2}$ such that $|\Gamma_{G'}(S)|<(k+1)|S|$. Let $R=S\cup \Gamma_{G'}(S)$. Then $$|R|\leq |S|+|\Gamma_{G'}(S)|< (k+2)|S|\leq r.$$ Furthermore, by the minimum degree condition, the number of edges in $G'[R]$ is at least $$\frac{1}{2}|S|\cdot 2\rho(k+2)> \rho |R|,$$ which is a contradiction.
	\end{proof}
	
	Now we are ready to state our first tree counting lemma.
	
	\begin{lemma} \label{lem:find many trees}
		For any $\rho>1$ and positive integer $k$, there exists $c_0=c_0(\rho,k)>0$ such that the following holds for every $n\geq 8$. Let $G$ be an $n$-vertex $(\rho,r)$-sparse graph with average degree at least $d\geq c_0^{-1}$. Let $T$ be a tree with $t\leq \frac{r}{2(k+2)}$ vertices and maximum degree at most $k$. Then $G$ contains at least $(c_0 d)^{t-1}$ copies of $T$.
	\end{lemma}
	
	\begin{proof}
		We show that $c_0=\frac{1}{16\rho(k+2)}$ suffices. Let $p=\frac{8\rho(k+2)}{d}<1$, and sample each edge of $G$ independently with probability $p$. Let the resulting graph be $G'$, and let $X=e(G')$. Then $\mathbb{E}(X)=pe(G)\geq \frac{pdn}{2}=4\rho (k+2)n$. As $X$ is the sum of independent indicator random variables, we can use Lemma \ref{lemma:chernoff} to write
		$\mathbb{P}\left(X\leq \frac{1}{2}\mathbb{E}(X)\right)\leq e^{-\frac{\mathbb{E}(X)}{8}}<\frac{1}{2}.$
		Hence, with probability at least $\frac{1}{2}$, $G'$ has average degree at least $4\rho(k+2)$. If this happens, we can apply Lemma \ref{lem:find one tree} to conclude that $G'$ contains a copy of $T$. Thus, the expected number of copies of $T$ in $G'$ is at least $\frac{1}{2}$. On the other hand, writing $N$ for the number of copies of $T$ in $G$, we also have that the expected number of copies of $T$ in $G'$ is $p^{t-1}N$. Hence, we get the inequality $p^{t-1}N\geq \frac{1}{2}$, which implies that $G$ contains at least $$N\geq \frac{1}{2}p^{-(t-1)}=\frac{1}{2}\left(\frac{d}{8\rho(k+2)}\right)^{t-1}>(c_0d)^{t-1}$$ copies of $T$.
	\end{proof}
	
	Now let us turn to our upper bound on the number of copies of a tree. For simplicity, we only present a counting result in case the tree is a caterpillar. However, it seems likely that a similar result should hold for trees in general as well.
	
	\begin{lemma}\label{lemma:num_trees_upperbound}
		Let $G$ be a graph with $n$ vertices and $m$ edges. Let $T$ be a caterpillar with $a$ non-leaf vertices and maximum degree $k$. Then $G$ contains at most $n\cdot (\frac{2m}{a})^{ak}$ copies of $T$.
	\end{lemma}
	
	\begin{proof}
		Let $d_1\geq d_2\geq\dots\geq d_n$ be the degree sequence of $G$. As $T$ is a caterpillar, its non-leaf vertices form a path on $a$ vertices, so let us first count the number of such paths in $G$.
		
		\begin{claim}
			For every vertex $v\in V(G)$, the number of paths on $a$ vertices in $G$ starting from $v$ is at most $d_1\dots d_{a-1}$.
		\end{claim}
		
		\begin{proof}
			We prove this by induction on $a$. If $a=2$, this is trivial, so let us assume that $a\geq 3$. Let $G'$ be the subgraph of $G$ we get after removing $v$, and let $d_1'\geq \dots\geq d_{n-1}'$ be the degree sequence of $G'$. There are $\deg_{G}(v)$ ways to choose the neighbour of $v$ in the path. If this neighbour is $v'\in V(G')$, we can use our induction hypothesis to conclude that there are at most $d_1'\dots d_{a-2}'$ paths on $a-1$ vertices in $G'$ starting with $v'$. Hence, the number of paths on $a$ vertices in $G$ starting with $v$ is at most $d_{G}(v)\cdot (d_1'\dots d'_{a-2})\leq d_1\dots d_{a-1}$, finishing the proof.
		\end{proof}
		
		Hence, the number of ways to embed the non-leaf vertices of $T$ is at most $n\cdot d_1\dots d_{a-1}$. Suppose that the non-leaf vertices of $T$ are already embedded in $G$, and their images are $v_1,\dots,v_a$. Then the number of ways to choose the leaves of $T$ is at most $$\deg_{G}(v_1)^{k-1}\dots\deg_{G}(v_a)^{k-1}\leq (d_1\dots d_{a})^{k-1}.$$
		Therefore, the total number of copies of $T$ in $G$ is at most $$n\cdot (d_1\dots d_{a})^{k}\leq n\cdot \left(\frac{d_1+\dots+d_{a}}{a}\right)^{ak}\leq  n\cdot \left(\frac{2m}{a}\right)^{ak},$$ where the first inequality is due to the AM-GM inequality. 
	\end{proof}
	
	\subsection{Piecing the trees together}
	
	In this section, we present our main technical lemma, which implies both Theorems \ref{thm:polylog_average_deg} and \ref{thm:bounded_size_average_deg}  after substituting the right parameters. Before we state this lemma, we show that if a graph $G$ contains many copies of a balanced caterpillar with the same set of leaves, then $G$ cannot be $(\rho,r)$-sparse. Recall that if $T$ is a tree with $t$ vertices and $a$ non-leaf vertices, then $\rho_{T}=\frac{t-1}{a}$. 
	
	\begin{lemma}\label{lemma:no_sparse}
		Let $\rho>0$ and let $k$ be a positive integer. Let $T$ be a balanced caterpillar with $t$ vertices, $a$ non-leaf vertices, and maximum degree at most $k$. Assume that $t< r\leq n$ and $\rho\leq (1-\frac{t}{r-t})\rho_T$. 
		Let $G$ be an $n$-vertex graph containing at least $r(\frac{2\rho r}{a})^{ak}$ copies of $T$ with the same set of leaves. Then $G$ is not $(\rho,r)$-sparse.
	\end{lemma}
	
	\begin{proof}
		Assume, for contradiction that $G$ is $(\rho,r)$-sparse. Let $\cT$ be a collection of at least $r(\frac{2\rho r}{a})^{ak}$ copies of $T$ in $G$ with the same set of leaves. Let $R_0\subset V(G)$ be the set of vertices spanned by the elements of $\cT$. First, observe that we must have $|R_0| \geq r$. Otherwise, as $G$ is $(\rho,r)$ sparse, $G[R_0]$ has at most $m=\rho |R_0|< \rho r$ edges, so by Lemma \ref{lemma:num_trees_upperbound}, it contains less than 
		$r(\frac{2\rho r}{a})^{ak}$ copies of $T$, a contradiction.
		
		Therefore, we can take a subcollection $\cT'\subset \cT$ such that the union of the trees in $\cT'$ spans at least $r-t$ and at most $r$ vertices. Let $R$ be the set of vertices in $G$ spanned by the union of the trees in $\cT'$. By Lemma \ref{lem:union of balanced trees}, $e(G[R])\geq (|R|-q)\rho_T$, where $q$ is the number of leaves in $T$. Hence,
		$$\frac{e(G[R])}{|R|}\geq \frac{|R|-q}{|R|}\rho_T\geq \frac{r-t-q}{r-t}\rho_T\geq \frac{r-2t}{r-t}\rho_T=\left(1-\frac{t}{r-t}\right)\rho_T\geq \rho.$$ Since $|R|\leq r$, this contradicts the assumption that $G$ is $(\rho,r)$-sparse, and the proof is complete.
	\end{proof}
	
	Now we are ready to state the promised main technical lemma.

	\begin{lemma}\label{lem:four conditions}
		Let $\rho>1$ and let $c_0=c_0(\rho,\lceil 2\rho\rceil+1)$ given by Lemma \ref{lem:find many trees}. Let $G$ be an $n$-vertex graph with average degree $d\geq c_0^{-1}$. Assume that there are positive integers $r$, $t$ and $a$ such that the following inequalities are satisfied.
		\begin{enumerate}
			\item $\rho\leq (1-\frac{t}{r-t})\frac{t-1}{a}$, \label{eqn:1}
			\item $\frac{t}{a}\leq 2\rho$, \label{eqn:2}
			\item $2(\lceil 2\rho\rceil+3)t\leq r\leq n$ and \label{eqn:3}
			\item $(c_0 d)^{t-1}\geq \binom{n}{t-a}\cdot r\cdot(\frac{2\rho r}{a})^{3t}$. \label{eqn:4}
		\end{enumerate}
		Then $G$ is not $(\rho,r)$-sparse.
	\end{lemma}
	
	\begin{proof}
		Conditions \ref{eqn:1} and \ref{eqn:3} imply that $t-1>a$, so Lemma \ref{lem:exists balanced tree} shows that there is a balanced caterpillar $T$ with $a$ non-leaf vertices, $t-1$ edges and maximum degree at most $k=\lceil t/a\rceil+1\leq \lceil 2\rho\rceil+1$. Note that $ka<3t$. Assume that $G$ is $(\rho,r)$-sparse. Then it follows by Lemma \ref{lem:find many trees} that $G$ contains at least $(c_0d)^{t-1}$ copies of $T$. Since $T$ has $t-a$ leaves and $G$ has 
		$\binom{n}{t-a}$ subsets of size $t-a$, it follows from condition~\ref{eqn:4} and the pigeonhole principle that there is a collection of at least $r(\frac{2\rho r}{a})^{3t}>r\cdot(\frac{2\rho r}{a})^{ka}$ copies of $T$ in $G$ which share the same set of leaves. Note that $T$ has $t-1$ edges and $a$ non-leaf vertices, so $\rho_T=\frac{t-1}{a}$ and condition \ref{eqn:1} gives $\rho\leq (1-\frac{t}{r-t})\rho_T$. Hence, we can apply Lemma \ref{lemma:no_sparse} to conclude that $G$ is not $(\rho,r)$-sparse, a contradiction.
	\end{proof}
	
	\subsection{Completing the proofs}
	
	In this section, we put everything together to conclude the proofs of our main results. First, we prove Theorem \ref{thm:polylog_average_deg} in the following equivalent form.
	
	\begin{theorem}
		For every $\rho>1$, there is a constant $C=C(\rho)$ such that the following holds for every sufficiently large $d$. Let $G$ be an $n$-vertex graph with average degree at least $d$, where $d\leq n^{\frac{\rho-1}{\rho}}$. Then $G$ is not $(\rho,nd^{-\frac{\rho}{\rho-1}}(\log d)^C)$-sparse.
	\end{theorem}
	
	Note that this indeed implies Theorem \ref{thm:polylog_average_deg} with $s=2\rho$.	
	
	\begin{proof}
		Let $C$ be sufficiently large with respect to $\rho$, let $d$ be sufficiently large with respect to $\rho$ and $C$, let $f=(\log d)^{C}$ and let $r=\lfloor nd^{-\frac{\rho}{\rho-1}}f\rfloor$. Then $r\leq n$, and by the conditions of the theorem, $r\geq \lfloor (\log d)^{C}\rfloor$. Let $\varepsilon=\frac{1}{\log d}$, $t=\lceil \frac{r\varepsilon}{8}\rceil$ and let $a=\lceil \frac{t}{\rho}(1-\varepsilon)\rceil$. Then $t\geq 20\log d$, assuming that $C$ and $d$ are sufficiently large. It suffices to prove that the four conditions in Lemma \ref{lem:four conditions} are satisfied.
		
		Note that 
		\begin{align*}
			\frac{a\rho}{t-1}&\leq \frac{t}{\rho}\cdot\left(1-\frac{\varepsilon}{2}\right)\cdot \frac{\rho}{t-1}=\frac{t}{t-1}\left(1-\frac{\varepsilon}{2}\right)\\
			&\leq \left(1+\frac{1}{10\log d}\right)\left(1-\frac{\varepsilon}{2}\right)\leq 1-\frac{\varepsilon}{4}\leq 1-\frac{t}{r-t},
		\end{align*}
		where the last inequality follows from  $t=\lceil \frac{r\varepsilon}{8}\rceil$. Hence, we have $\rho\leq (1-\frac{t}{r-t})\frac{t-1}{a}$ and condition \ref{eqn:1} is satisfied.
		
		Conditions \ref{eqn:2} and \ref{eqn:3} are immediate from the definitions of $t$ and $a$. Hence, it remains to verify condition \ref{eqn:4}, that is
		\begin{equation}\label{equ:cond4}
		\left(c_0d\right)^{t-1}\geq \binom{n}{t-a}\cdot r\cdot \left(\frac{2\rho r}{a}\right)^{3t}.
		\end{equation}	
		First, since $a<\frac{t}{\rho}$, note that
			\begin{align*}
				\binom{n}{t-a}\leq \left(\frac{en}{t-a}\right)^{t-a}\leq \left(\frac{en}{t-t/\rho}\right)^{t-a}\leq \left(\frac{e}{1-1/\rho}\right)^t\cdot \left(\frac{n}{t}\right)^{t-a}.
			\end{align*}
			Also,
			\begin{align*}
				\left(\frac{n}{t}\right)^{t-a}
				&\leq
				\left(\frac{8n}{\eps r}\right)^{t-a}\leq
				\left(\frac{9d^{\frac{\rho}{\rho-1}}}{\varepsilon f}\right)^{t-a}\leq \left(\frac{9d^{\frac{\rho}{\rho-1}}}{\varepsilon f}\right)^{t-\frac{t}{\rho}+\frac{t\varepsilon}{\rho}}
				\leq \left(9 d^{1+\frac{\varepsilon}{(\rho-1)}}\cdot(\varepsilon f)^{-1+\frac{1}{\rho}-\frac{\varepsilon}{\rho}}\right)^{t}.
			\end{align*}
			Using that $d^{\varepsilon}=d^{1/\log d}=e$, this implies that
			$$\binom{n}{t-a}\leq \big(c_1d\cdot (\varepsilon f)^{-1+\frac{1-\varepsilon}{\rho}}\big)^{t}$$
			holds for some $c_1=c_1(\rho)>0$. On the other hand, we have $(\frac{2\rho r}{a})^{3t}\leq (\frac{c_2}{\varepsilon^{3}})^{t}$ for some $c_2=c_2(\rho)$. Hence, in order to prove (\ref{equ:cond4}), it suffices to show that
			\begin{equation}
				\label{xeq1}
				(c_0d)^{1-\frac{1}{t}}\geq c_1 c_2\cdot d\cdot r^{\frac{1}{t}}\cdot \varepsilon^{-4+\frac{1-\varepsilon}{\rho}}\cdot  f^{-1+\frac{1-\varepsilon}{\rho}}.
			\end{equation}
			As $t>10\log d$ and $d$ is sufficiently large, we have $(c_0d)^{1-\frac{1}{t}}>\frac{c_0d}{2}$. Also, $t\geq \frac{r}{8\log d}>\log r$, so $r^{\frac{1}{t}}<3$. Finally, recalling that $f=(\log d)^{C}$ and $\varepsilon=\frac{1}{\log d}$, we get that (\ref{xeq1}) holds whenever $C$ and $d$ are sufficiently large in terms of $\rho$. This completes the proof of the theorem.
	\end{proof}

	Finally, we prove the following equivalent version of Theorem \ref{thm:bounded_size_average_deg}.
	
	\begin{theorem} \label{thm:bounded size}
		For every $\rho>1$ and $0<\eps<1$, there is a positive integer $r$ such that the following holds for all sufficiently large $n$. Let $G$ be an $n$-vertex graph with average degree $d\geq n^{1-\frac{1}{\rho}+\eps}$. Then $G$ is not $(\rho,r)$-sparse.
	\end{theorem}
	
	\begin{proof}
		Assume that $r$ is sufficiently large in terms of $\rho$ and $\eps$. Let $t=\lceil \frac{4\rho}{\eps} \rceil$ and let $a=\lfloor \frac{t}{\rho}\rfloor -1$. It suffices to prove that the four conditions in Lemma \ref{lem:four conditions} are satisfied.
		
		Note that $$\frac{a\rho}{t-1}\leq \frac{(t/\rho-1)\rho}{t-1}=\frac{t-\rho}{t-1}=1-\frac{\rho-1}{t-1}\leq 1-\frac{t}{r-t},$$ where the last inequality follows from the assumption that $r$ is sufficiently large in terms of $\rho$ and $\eps$. Hence, we have $\rho\leq (1-\frac{t}{r-t})\frac{t-1}{a}$ and condition \ref{eqn:1} is satisfied.
		
		Condition \ref{eqn:3} is immediate from the definition. Also by definitions of $a$ and $t$, we have $\frac{t}{\rho} \geq 4$ and hence $a \geq \frac{t}{\rho}-2 \geq \frac{t}{2\rho}$, verifying condition \ref{eqn:2}.
		
		Now let us verify condition \ref{eqn:4}. We have $d\geq n^{1-\frac{1}{\rho}+\eps}$ and $\binom{n}{t-a}\leq n^{t-a}$. Hence, it is enough to prove that $$(c_0n^{1-\frac{1}{\rho}+\eps})^{t-1}\geq n^{t-a}\cdot r\cdot \left(\frac{2\rho r}{a}\right)^{3t}.$$ Note that $n$ is sufficiently large compared to the other parameters, so it suffices to prove the corresponding inequality for the exponents of $n$ on both sides, i.e., that  $(1-\frac{1}{\rho}+\eps)(t-1)>t-a$. By the definition of $t$, $\eps t \geq 4\rho>4$. Hence
		$$\left(1-\frac{1}{\rho}+\eps\right)(t-1)=(t-1)-\frac{t-1}{\rho}+\eps (t-1)\geq t-\frac{t}{\rho}+\eps t-2\geq t-a+\eps t-4>t-a,$$ and the proof is complete.
	\end{proof}

	\section{Concluding remarks}
	
	\subsection{Regular subgraphs}
	
	In this paper, we proved optimal bounds on the order of the smallest subgraph $G[R]$ of average degree at least $s$ in a graph $G$ of given order and average degree. As mentioned in the introduction, in case $s$ is an integer, we believe that the strengthening of our results in which the average degree of $G[R]$ is replaced with its minimum degree should also hold. Moreover, one can further strengthen this by requiring that $G[R]$ contains an $s$-regular subgraph.
	
	\begin{conjecture} \label{conj1}
		For every integer $s\geq 3$, there is a constant $C=C(s)$ such that the following holds for every sufficiently large $n$. Let $G$ be an $n$-vertex graph with average degree at least $d$, where $C\log \log n\leq d\leq n^{\frac{s-2}{s}}$. Then $G$ contains an $s$-regular subgraph on at most $nd^{-\frac{s}{s-2}}(\log n)^C$ vertices.
	\end{conjecture}
	
	Note that in the special case where $d\leq (\log n)^{C(s-2)/s}$, Conjecture \ref{conj1} asserts the existence of an $s$-regular subgraph, without a requirement on its order. This is the well studied Erd\H os--Sauer problem which was resolved very recently in \cite{JS22}. It was shown there that for large $C=C(s)$, every $n$-vertex graph with average degree at least $C\log \log n$ has an $s$-regular subgraph. This is tight, as an old construction of Pyber, R\"odl and Szemer\'edi \cite{PRSz95} shows that there is some $c>0$ such that there are $n$-vertex graphs with average degree at least $c\log \log n$ and no $s$-regular subgraph.
	
	On the other hand, when $d$ is very large and we are looking for a subgraph of bounded order, we have the following conjecture generalizing Theorem~\ref{thm:bounded_size_average_deg}. This conjecture quantifies Problem 7.1 from \cite{JN17}.
	
	\begin{conjecture} \label{conj2}
		For every integer $s\geq 3$ and $\eps>0$, there is a positive integer $t$ such that the following holds for all sufficiently large $n$. Let $G$ be an $n$-vertex graph of average degree at least $ n^{1-\frac{2}{s}+\eps}$. Then $G$ contains an $s$-regular subgraph on at most $t$ vertices.
	\end{conjecture}
	
	As we remarked earlier, Conjecture \ref{conj2} is known to be true if $s$ is even \cite{J20}, or $s=3$ \cite{J21}.
	
	\subsection{Uniform hypergraphs}
	
	Another interesting direction one may explore is the analogous question for uniform hypergraphs, which was also considered by Feige and Wagner \cite{FW16}.
	
	\begin{problem}\label{problem:hyp}
		Let $r,n$ be positive integers, and $s,d>1$ be real numbers. Determine the asymptotic value of the smallest $t=t_r(n,d,s)$ such that every $r$-uniform hypergraph on $n$ vertices of average degree at least $d$ contains a subhypergraph on at most $t$ vertices of average degree at least $s$. 
	\end{problem}
	
	This problem is closely related to another well known conjecture of Feige \cite{F08} about \emph{even covers} of hypergraphs. An even cover of a hypergraph is a non-empty subhypergraph in which each vertex is contained in an even number of edges. Feige conjectured that every $r$-uniform hypergraph with $n$ vertices and average degree $d$ contains an even cover on at most $nd^{-\frac{2}{r-2}}\mbox{polylog}(n)$ vertices, which was recently settled by Guruswami, Kothari, and Manohar \cite{GKM}. Indeed, in order to find a small even cover, one needs to first guarantee a small subhypergraph of average degree at least~2.

	\vspace{0.3cm}
	\noindent	
	{\bf Acknowledgements.}
	We would like to thank Noga Alon for valuable discussions.

\end{document}